\documentclass[a4paper]{article}
\usepackage{amsmath}
\usepackage{amssymb,amsthm}
\usepackage{graphics}
\usepackage{caption}
\usepackage{xcolor}
\usepackage{pgf}

\usepackage{amsfonts}

\allowdisplaybreaks

\newcommand{\pp}{\mathbb{P}}

\newcommand{\ee}{\mathbb{E}\,}
\newcommand{\ii}{\mathcal{I}}

\newcommand{\rr}{\mathbb{R}}

\newcommand{\ch}{\mathcal{H}}

\renewenvironment{proof}[1][\proofname]{\par \normalfont \trivlist
\item[\hskip\labelsep\itshape #1]\ignorespaces
}{%
\hspace*{\fill}$\Box$ \endtrivlist }
\renewcommand{\proofname}{{\bf Proof}}

\makeatletter
\def\newrmtheorem#1{\@ifnextchar[{\@rmothm{#1}}{\@rmnthm{#1}}}

\def\@rmnthm#1#2{%
\@ifnextchar[{\@rmxnthm{#1}{#2}}{\@rmynthm{#1}{#2}}}

\def\@rmxnthm#1#2[#3]{\expandafter\@ifdefinable\csname #1\endcsname
{\@definecounter{#1}\@addtoreset{#1}{#3}%
\expandafter\xdef\csname the#1\endcsname{\expandafter\noexpand
  \csname the#3\endcsname \@rmthmcountersep \@rmthmcounter{#1}}%
\global\@namedef{#1}{\@rmthm{#1}{#2}}\global\@namedef{end#1}{\@endrmtheorem}}}

\def\@rmynthm#1#2{\expandafter\@ifdefinable\csname #1\endcsname
{\@definecounter{#1}%
\expandafter\xdef\csname the#1\endcsname{\@rmthmcounter{#1}}%
\global\@namedef{#1}{\@rmthm{#1}{#2}}\global\@namedef{end#1}{\@endrmtheorem}}}

\def\@rmothm#1[#2]#3{\expandafter\@ifdefinable\csname #1\endcsname
  {\global\@namedef{the#1}{\@nameuse{the#2}}%
\global\@namedef{#1}{\@rmthm{#2}{#3}}%
\global\@namedef{end#1}{\@endrmtheorem}}}

\def\@rmthm#1#2{\refstepcounter
    {#1}\@ifnextchar[{\@rmythm{#1}{#2}}{\@rmxthm{#1}{#2}}}

\def\@rmxthm#1#2{\@beginrmtheorem{#2}{\csname the#1\endcsname}\ignorespaces}
\def\@rmythm#1#2[#3]{\@opargbeginrmtheorem{#2}{\csname
       the#1\endcsname}{#3}\ignorespaces}

\def\@rmthmcounter#1{\noexpand\arabic{#1}}
\def\@rmthmcountersep{}
\def\@beginrmtheorem#1#2{\rm \trivlist
      \item[\hskip \labelsep{\bf #1\ #2\thmrmcounterend}]}
\def\@opargbeginrmtheorem#1#2#3{\rm \trivlist
      \item[\hskip \labelsep{\bf #1\ #2\ (#3)\thmrmcounterend}]}
\def\@endrmtheorem{\endtrivlist}
\def\thmrmcounterend{\hskip 0em\relax}
\def\newrmwntheorem#1#2{\expandafter\@ifdefinable\csname #1\endcsname%
\global\@namedef{#1}{\@rmwnthm{#1}{#2}}%
\global\@namedef{end#1}{\@endrmwntheorem}}

\def\newsltheorem#1{\@ifnextchar[{\@slothm{#1}}{\@slnthm{#1}}}

\def\@slnthm#1#2{%
\@ifnextchar[{\@slxnthm{#1}{#2}}{\@slynthm{#1}{#2}}}

\def\@slxnthm#1#2[#3]{\expandafter\@ifdefinable\csname #1\endcsname
{\@definecounter{#1}\@addtoreset{#1}{#3}%
\expandafter\xdef\csname the#1\endcsname{\expandafter\noexpand
  \csname the#3\endcsname \@slthmcountersep \@slthmcounter{#1}}%
\global\@namedef{#1}{\@slthm{#1}{#2}}\global\@namedef{end#1}{\@endsltheorem}}}

\def\@slynthm#1#2{\expandafter\@ifdefinable\csname #1\endcsname
{\@definecounter{#1}%
\expandafter\xdef\csname the#1\endcsname{\@slthmcounter{#1}}%
\global\@namedef{#1}{\@slthm{#1}{#2}}\global\@namedef{end#1}{\@endsltheorem}}}

\def\@slothm#1[#2]#3{\expandafter\@ifdefinable\csname #1\endcsname
  {\global\@namedef{the#1}{\@nameuse{the#2}}%
\global\@namedef{#1}{\@slthm{#2}{#3}}%
\global\@namedef{end#1}{\@endsltheorem}}}

\def\@slthm#1#2{\refstepcounter
    {#1}\@ifnextchar[{\@slythm{#1}{#2}}{\@slxthm{#1}{#2}}}

\def\@slxthm#1#2{\@beginsltheorem{#2}{\csname the#1\endcsname}\ignorespaces}
\def\@slythm#1#2[#3]{\@opargbeginsltheorem{#2}{\csname
       the#1\endcsname}{#3}\ignorespaces}

\def\@slthmcounter#1{.\noexpand\arabic{#1}}
\def\@slthmcountersep{}
\def\@beginsltheorem#1#2{\sl \trivlist
      \item[\hskip \labelsep{\bf #1\ #2\thmslcounterend}]}
\def\@opargbeginsltheorem#1#2#3{\sl \trivlist
      \item[\hskip \labelsep{\bf #1\ #2\ (#3)\thmslcounterend}]}
\def\@endsltheorem{\endtrivlist}
\def\thmslcounterend{\hskip 0em\relax}
\def\newslwntheorem#1#2{\expandafter\@ifdefinable\csname #1\endcsname%
\global\@namedef{#1}{\@slwnthm{#1}{#2}}%
\global\@namedef{end#1}{\@endslwntheorem}}

\makeatother

\newsltheorem{theorem}{Theorem}[section]
\newsltheorem{proposition}[theorem]{Proposition}
\newsltheorem{lemma}[theorem]{Lemma}
\newsltheorem{corollary}[theorem]{Corollary}

\newrmtheorem{definition}[theorem]{Definition}
\newrmtheorem{example}[theorem]{Example}
\newrmtheorem{remark}[theorem]{Remark}
\newrmtheorem{problem}[theorem]{Problem}
\newrmtheorem{algorithm}[theorem]{Algorithm}
\newrmtheorem{cond}[theorem]{Condition}
\newrmtheorem{convention}[theorem]{Convention}
\newrmtheorem{assumption}[theorem]{Assumption}

\begin{document}

\providecommand{\keywords}[1]
{
  \small	
  \textsl{Keywords:} #1
}

\providecommand{\ams}[1]
{
  \small	
  \textsl{AMS subject classification:} #1
}

\title{Approximation of nonnegative systems by moving averages of fixed order}%

\author{Lorenzo Finesso\thanks{Lorenzo Finesso is with the Institute of Electronics, Computer, and Telecommunication Engineering,
National Research Council, CNR-IEIIT, Padova; email: {\tt lorenzo.finesso@ieiit.cnr.it}} \and
        Peter Spreij\thanks{Peter Spreij is with the Korteweg-de Vries Institute for Mathematics,
Universiteit van Amsterdam and with the  Institute for Mathematics, Astrophysics and Particle Physics, Radboud University, Nijmegen; e-mail: {\tt spreij@uva.nl}}
}

\maketitle

\begin{abstract}
We pose the  approximation problem for scalar nonnegative input/\-out\-put systems via impulse response convolutions of finite order, i.e.\ finite order moving averages, based on repeated observations of input/output signal pairs. The problem is  converted into a nonnegative matrix factorization with special structure for which we use Csisz\'ar's I-divergence as the criterion of optimality. Conditions are given, on the input/output data, that guarantee the existence and uniqueness of the minimum. We propose an algorithm of the alternating minimization type for I-divergence minimization, and present its asymptotic behavior. For the case of noisy observations we give the large sample properties of the statistical version of the minimization problem for different observation regimes. Numerical experiments confirm the asymptotic results and exhibit fast convergence of the proposed algorithm.


\ams{62B10, 62E20, 94A17, 93B30, 93E12}

\end{abstract}

\section{Introduction}\label{section:intro}

%
In this paper we pose the problem of the time-domain approximation of nonnegative input/output systems by finite (nonnegative) impulse response convolutions of fixed order $q$,  when input/output observations are available. In principle, the order $q$ ia a low number compared to the number of observations,
We propose an iterative algorithm to find the best approximation, and study the asymptotical behavior of the algorithm. The present paper is a variation on and complements \cite{fs2015}, where the order of the convolution was not fixed, but varies with the sample size. Contrary to contributions prior to, but in line with \cite{fs2015}, our treatment allows for $m>1$  input/\-output pairs. This setting leads easily to a statistical analysis when the output is observed with noise. We then study large sample properties of the resulting parameter estimators
when (1) the number of input/\-output pairs $m$ grows unboundedly but the time horizon is fixed, (2) the number of observations, the time horizon, $N$, tends to infinity, but $m$ is fixed and (3) a mixture of the previous two cases. It is noted that the last two cases are not meaningful when the order of the convolution is not fixed, as in \cite{fs2015}. Indeed, fixing the order $q$ is the main difference with our earlier contributions.
Similar algorithms to ours for the case $m=1$ has been studied in~\cite{snyderetal1992} and \cite{Vardietal1985}. Following the choice made in those early contributions our criterion of optimality will be Csisz\'ar's I-divergence, which as argued in~\cite{cs1991} (see also \cite{snyderetal1992}), is the best choice for approximation problems under nonnegativity constraints.

We emphasize that our approach to the approximation of a given input/\-out\-put system by a linear time invariant system is different from the usual identification or realization of (nonnegative) linear systems, see~\cite{benvenutifarina2004} for a survey, and for instance~\cite{nagy2007}, \cite{nagy2005}, \cite{gurvits2007}, \cite{shu2008}, \cite{andersondeistler1996}, \cite{farina1995}. From the mathematical point of view, the techniques that we have used in~\cite{fs2006} to analyse a nonnegative matrix factorization algorithm have been shown to be very useful in the present context as well, as demonstrated in \cite{fs2015},  and provided several benefits over the analyses contained in~\cite{snyderetal1992}. We will provide explicit conditions for the existence and uniqueness of the minimizer of the criterion  in terms of the data. The algorithm that minimizes the informational divergence criterion is of the same alternating minimization type as in \cite{fs2015}, and the optimality conditions (the Pythagorean relations) are shown satisfied at each step. As demonstrated in \cite{fs2015}, these are the core of a proof of convergence which is more transparent than other proofs in the literature, e.g.~\cite{cover1984}, \cite{snyderetal1992}, and \cite{Vardietal1985}.

A different theoretical approach, in the frequency domain, has been followed in~\cite{liubauer2010a} and in~\cite{liubauer2010b}. The contributions of the present paper are theoretical, possible applications of the algorithm are e.g.\ in the fields of image processing, emission tomography, industrial processes,  charge routing networks, compartmental systems, storage systems. For these we refer for instance to~\cite{dewasurendra2007,snyderetal1992,sullivan2007,Vardietal1985,farinarinaldibook} and references therein.

A brief summary of the paper follows. In Section~\ref{section: problem} we state the problem and formulate conditions for strict convexity of the objective function, and hence for the existence and uniqueness of the solution. In Section~\ref{sec:lift} the original problem is lifted into a higher dimensional setting, thus making it amenable to alternating minimization. The optimality properties (Pythagoras rules) of the ensuing partial minimization problems are recalled here. Then we derive the iterative minimization algorithm combining the solutions of the partial minimizations, we present its first properties and the important result on the convergence of the algorithm.
In Section~\ref{section:stats}, taking advantage of the repeated input/output measurements setup or the possibility of a growing time horizon, we give a concise treatment of a statistical version of the approximation problem, focusing on its large sample properties. In the last Section~\ref{section:numerics} we present numerical experiments that confirm the asymptotic results and exhibit fast convergence properties of the algorithm.

As mentioned above, the present paper is a follow up to \cite{fs2015}, where the dimension of the parameter was not fixed, but varies with the sample size. Accordingly we often only highlight the differences and refer, unless extra or different arguments are needed, to \cite{fs2015} for proofs. Main differences with \cite{fs2015} are in Section~\ref{section:stats}, where we discuss asymptotic properties of estimators under different observation regimes (including those where the time horizon tends to infinity), which can now be treated because of the fixed dimension of the parameter.

\section{Problem statement and preliminary results}\label{section: problem}
A discrete time, causal, convolutional moving average system $\mathcal S_h$ of order $q$  maps input sequences  $(u_t)_{t \in \mathbb N} \in  \mathbb R^\mathbb N$ into output sequences  $(y_t)_{t \in \mathbb N} \in  \mathbb R^\mathbb N$, and is completely characterized by an impulse response vector  $h=(h_t)_{t \in \{0,\ldots,q\}}$, such that
\begin{equation*} \label{convsys}
y_t =  \mathcal S_h u_t = \sum_{k=0}^t  h_k u_{t-k}, \qquad t \in \mathbb N,
\end{equation*}
where $h_k$ is set to zero for $k>q$. Alternatively, one can also write
\begin{equation} \label{convsysq}
y_t =  \mathcal S_h u_t = \sum_{k=0}^{t\wedge q}  h_k u_{t-k}, \qquad  t \in \mathbb N,
\end{equation}
where we write $t\wedge q$ for $\min\{t,q\}$.

Throughout the paper we consider a time horizon $N$ for which we assume $N\geq q$, a standing assumption. Hence we have to replace \eqref{convsysq} by
\begin{equation} \label{convsysqN}
y_t =  \mathcal S_h u_t = \sum_{k=0}^{t\wedge q}  h_k u_{t-k}, \qquad 0\leq t \leq N.
\end{equation}
The special case where $q=N$ has been treated in \cite{fs2015} and $N<q$ yields a redundancy, as the parameters $h_{N+1},\ldots,h_q$ don't play a role in \eqref{convsysqN}.
Rewriting Equation~\eqref{convsysqN} in matrix form, one gets the system of equations
\begin{equation}  \label{eq:matrix}
\begin{pmatrix}
y_0 \\
\vdots \\
\vdots \\
y_N
\end{pmatrix}
=
\begin{pmatrix}
h_0 & 0   & \cdots & \cdots & \cdots & 0 \\
\vdots &    \ddots      & \ddots   &  &   & \vdots \\
h_q &  & \ddots & \ddots & \\
0 & \ddots & & \ddots & \ddots & \vdots \\
\vdots & \ddots & \ddots & & \ddots & 0 \\
0 & \cdots & 0 & h_q & \cdots & h_0
\end{pmatrix}
\begin{pmatrix}
u_0 \\
\vdots \\
\vdots \\
u_N
\end{pmatrix},
\end{equation}
compactly written as
\begin{equation}\label{eq:yu}
y=T(h)u,
\end{equation}
having introduced the notations $u=(u_0,\ldots,u_N)^\top$, \,  $y=(y_0,\ldots,y_N)^\top$ and $T(h)\in\rr^{(N+1)\times (N+1)}$ for the matrix in~\eqref{eq:matrix}. For $m$ input sequences $u^j$, with corresponding output sequences $y^j$, where $j=1,\dots, m$, equation~\eqref{eq:yu} becomes
\begin{equation}\label{eq:YU}
Y=T(h)U,
\end{equation}
where $Y=(y^1,\ldots,y^m)\in\rr^{(N+1)\times m}$ and $U=(u^1,\ldots,u^m)\in\rr^{(N+1)\times m}$. Elements of $Y$ and $U$ are denoted $Y_{ij}$ and $U_{ij}$, instead of $y^j_i$ and $u^j_i$.

%

In many practical contexts the inputs and outputs $U$ and $Y$ are directly measured \emph{data}, while $h$ is not known or, more generally, a causal convolutional system $\mathcal S_h$ is not known to exist such that $Y=T(h)U$. In either of these cases an interesting problem is to find $h$ such that the approximate relation
\begin{equation} \label{eq:approx}
Y \approx T(h)U
\end{equation}
is the best possible with respect to a specified loss criterion.

In this paper we concentrate on this problem, under the extra condition that~\eqref{eq:approx} is the approximate representation of the behavior of a positive system, i.e.\ all quantities in~\eqref{eq:approx} are nonnegative real numbers. The goal is the determination of the \emph{best} nonnegative vector $h=(h_0,\ldots,h_q)^\top$, where the loss criterion, chosen to measure the discrepancy between the left and the right hand side in~\eqref{eq:approx}, is the {\em I-divergence} between nonnegative matrices. See~\cite{cs1991} for a justification from first principles.

For given nonnegative vectors, matrices, tensors $M$ and $N$ of the same size, indexed by some variable $\alpha$, the {\em I-divergence} between them is defined as
\begin{equation} \label{def:Idiv}
\ii(M||N):= \sum_{\alpha}\left(M_{\alpha}\log\frac{M_{\alpha}}{N_{\alpha}}-M_{\alpha}+N_{\alpha}\right)\leq\infty.
\end{equation}
In definition~(\ref{def:Idiv}) we also adopt the usual conventions $\frac{0}{0}=0$, $0\log 0 = 0$ and $\frac{p}{0}=\infty$ for $p>0$.

\begin{problem}\label{problem:minh}
For given $Y\geq 0$  and $U \geq 0 $, find a nonnegative vector $h=(h_0,\ldots,h_q)^\top\in \ch := \mathbb R^{q+1}_+$ such that $F:\ch\to [0,\infty]$,
\[
F(h):=\ii(Y||T(h)U)
\]
is minimized over $\ch$.
\end{problem}
Problem~\ref{problem:minh} is well posed if there exists at least one $h\in\rr^{q+1}_+$ such that $F(h)$ is finite.
Under a rather weak condition on the data $(U, Y)$, the loss $F(h)$ is strictly convex (and hence Problem~\ref{problem:minh} is well posed), a property that simplifies the study of the existence and uniqueness of the solution of Problem~\ref{problem:minh}.

\begin{cond} \label{cond:sc}
For all $i\in\{0,\ldots,N\}$ there exists $j\in\{1,\ldots,m\}$ such that $U_{0j}>0$ and $Y_{ij}>0$.
\end{cond}

\noindent
This condition holds e.g.\ under the (stronger) assumption that for some experiment $j$, with initial input $U_{0j}>0$, the output trajectory $Y_{ij}$ is strictly positive for all $i$. Lemma~\ref{lemma:fconvex} below is similar to \cite[Lemma II.6]{fs2015}, but its proof uses different arguments.

\begin{lemma}\label{lemma:fconvex}
Under Condition \ref{cond:sc} the loss $F(h)$ is strictly convex on its effective domain, i.e. the set
$\{\,h\in\ch: \, F(h)< \infty \, \}$.
\end{lemma}

\begin{proof}
%
We exploit strict concavity of the logarithm. It is sufficient to prove strict concavity of $h\mapsto \sum_{ij} Y_{ij}\log (T(h)U)_{ij}$. Note that all mappings $h\mapsto Y_{ij}\log (T(h)U)_{ij}$ are concave. Hence it is sufficient to show that at least on of them is strictly concave. Fix $i$ and choose $j=j(i)$ such that $Y_{ij}>0$ and $U_{0j}>0$. We show that for at least on pair $(i,j)$ one has strict concavity of $h\mapsto\log (T(h)U)_{ij}$. Choose different vectors $h^0,h^1\in\rr^{q+1}_+$ and let $\bar h=(1-t)h^0+th^1$ for $t\in(0,1)$. We have to show that there is an $i$ such that $(T(\bar h)U)_{ij}$ is not equal to one of the $(T(h^0)U)_{ij}$ and $(T(h^1)U)_{ij}$. Suppose on the contrary that $(T(\bar h)U)_{ij}=(T(h^0)U)_{ij}$ for all $i$. It is sufficient to restrict our attention to $i\leq q$, since we assumed $N\geq q$. This is then equivalent to $\sum_{l=0}^{q}v_lU_{i-l,j}=0$ for all $i$, where $v_l=h^1_l-h^0_l$. This gives a linear system of $q+1$ equations in the $v_l$ in which the coefficient matrix is lower triangular with the $U_{0,j(i)}$ on the diagonal. But these diagonal elements are all strictly positive, hence the $v_l$ are all zero, which contradicts $h^0\neq h^1$.
\end{proof}

\begin{remark}\label{remark:boundary}
In solving Problem~\ref{problem:minh}, minimizers $h^*$ at the boundary of $\mathcal H=\mathbb R^{q+1}_+$, i.e. with some zero components, are the rule rather than an exception when \emph{q=N}, see~\cite[Remark~10]{fs2015}. But, if $N$ is much larger than $q$, it has been observed that one often has interior solutions. See Section~\ref{section:numerics} for an illustration of this remark.
\end{remark}
We now state the existence and uniqueness result. The statement and its proof are verbatim the same as for Proposition~7 in~\cite{fs2015}. An important ingredient of the proof is that the search for a minimizer can be confined to a suitable compact set, on which the divergence is finite.
\begin{proposition}\label{proposition:minh}
Assume Condition~\ref{cond:sc} is satisfied, \ then Problem~\ref{problem:minh} admits a unique solution.
\end{proposition}
\begin{remark}\label{remark:uniqueh}
Suppose that given the input sequences, the outputs are obtained by  a true convolutional system $Y=T(h^*)U$ for some $h^*\in\mathcal{H}$. It follows from Proposition~\ref{proposition:minh} that under Condition~\ref{cond:sc}, the minimizer of $h\mapsto F(h)$ is $h^*$ and $F(h^*)=0$. Note too that under the same Condition~\ref{cond:sc} the system of equations $T(h)U=T(h^*)U$ has the unique solution $h=h^*$.

If for the general case one wants to check whether a proposed vector $h^*$ is a minimizer, it is by the convexity result of Lemma~\ref{lemma:fconvex} sufficient to check the Kuhn-Tucker conditions  (see e.g.~\cite[Theorem~2.19]{zangwill}).
\end{remark}

\section{The algorithm} \label{sec:lift}
\setcounter{equation}{0}

To solve Problem~\ref{problem:minh} we propose an alternating minimization algorithm, based on a variation of the lifting technique pioneered by \cite{ct1984}. The same approach was previously adopted in~\cite{fs2015} for the solution of Problem~\ref{problem:minh} under the condition $q=N$.
The results of this section are in spirit the same as the corresponding ones in~\cite[Section~III]{fs2015} and can be derived in an analogous way. Proofs are therefore omitted.


This leads to the following algorithm, almost identical to Algorithm~19 in~\cite{fs2015},  with minor differences only, see also Remark~\ref{moreq}.
\begin{algorithm}\label{algorithm:h}
Initialize at a strictly positive vector $h^0$ and define recursively for $t\geq 0$
$$
h^{t+1}=I(h^t),
$$
where the map $I$ acts on the components of $h^t$ as follows. For $k=0,\ldots,q$,
\begin{equation}\label{eq:rech}
h^{t+1}_k = I_k(h^t):=\frac{h^t_k}{\sum_{l=0}^{N-k}U_{l\centerdot}}\sum_{j=1}^m\sum_{i=k}^{N}\frac{Y_{ij}U_{i-k,j}}{\sum_{p=0}^{i\wedge q}h^t_{p}U_{i-p,j}}.
\end{equation}
\end{algorithm}
If the data satisfy $U_{0\centerdot}>0$, as is the case under Condition~\ref{cond:sc}, any $h^0>0$ componentwise is sufficient for $F(h^0)<\infty$.
%
%
%

\begin{remark} \label{moreq}
If $q=N$, Algorithm~\ref{algorithm:h} is exactly the same as Algorithm~19 in \cite{fs2015}.  If $q<N$,
one can add artificial parameters $h_k=0$ for $k=q+1,\ldots,N$. Starting the algorithm in $h_k^0=0$ for those $k$, we see that all iterated values $h^t_k$ are zero as well. The extension of the algorithm with these iterates then also yields the algorithm of \cite{fs2015}, with the modification of the zero initial values for $k=q+1,\ldots,N$. Note that, although Algorithm~\ref{algorithm:h} can this be viewed as a special case of \cite[Algorithm~19]{fs2015}, it requires separate derivation in principle. The reason is that the algorithm follows from the two partial minimization problems. The second one of which can be considered as a constraint version of the second minimization problem in \cite{fs2015}, the constraints being $h_k=0$ for $k=q+1,\ldots,N$. Curiously enough the solution of the second minimization problem coincides with the $h^*_k$ in the non-constrained problem in \cite{fs2015}.
\end{remark}
Here are a few properties, parallelling those in \cite{fs2015}. Positivity of the initial values is preserved by the iterations; the algorithm decreases the divergence $\ii(Y||T(h^t)U)$ at each step; the recursion enjoys a stability property, if $h^t$ is such that $F$ is increasing (decreasing) in the $k$-th coordinate of $h^t$, then $h^{t+1}_k<h^t_k$ ($h^{t+1}_k>h^t_k$); the vectors $h^t$ belong to a certain compact set, in fact a simplex.

\begin{remark}
Algorithm~\ref{algorithm:h} has multiplicative update rules for the $h^t_k$ and all iterates remain positive. In principle the algorithm risks to get trapped if some component $h_k^t$ is (nearly) zero. But Theorem~\ref{thm:limit} below guarantees that the algorithm converges to the minimizing $h$, and hence will not get trapped elsewhere.
This is in contrast with many other algorithms with a multiplicative update rule. For further discussion on this issue see~\cite{lin2007}.
\end{remark}
We close this section with a result concerning the asymptotic behaviour of Algorithm~\ref{algorithm:h}, Theorem~\ref{thm:limit} below.
Its omitted proof, much like the one of Theorem~25 in \cite{fs2015}, heavily relies on the optimality results for the partial minimizations and a series of lemmas as in the cited paper.
\begin{theorem}\label{thm:limit}
The sequence of iterates $h^t$ converges to a limit $h^\infty$ which minimizes $h\to\ii(Y||T(h)U)$.
\end{theorem}
Here is a a very simple example.
\begin{example}
Suppose $q=0$ and $N\geq 0$. This is an instance in which Problem~\ref{problem:minh} has an explicit solution, $h^*_0=\frac{\sum_{ij}Y_{ij}}{\sum_{ij}U_{ij}}$. Starting with $h^0_0>0$, Algorithm~\ref{algorithm:h} produces $h^1_0=h^*_0$, so it reaches the minimizing value in one step. When $q>0$ there is no termination of the algorithm in finitely many steps that achieves the minimizing vector $h^*$, but an explicit solution for $q=1$, $N=1$ is available, see \cite[Example~II.9]{fs2015}. Depending on the data, there are boundary solutions in the latter case.
\end{example}

\section{Statistics} \label{section:stats}
\setcounter{equation}{0}

In the previous sections we focused on the minimization of $F(h)=\ii(Y||T(h)U)$ over $h\in\rr^{q+1}_+$, where $Y$ and $U$ were given matrices and we presented an algorithm that asymptotically yields the minimizer. In the present section we concentrate on a statistical version of the minimization problem and its large sample properties under different observation regimes. Specifically we study the asymptotics when the number of input sequences grows and/or the time horizon tends to infinity. The latter is possible in the present context, unlike in \cite{fs2015}, because of the fixed dimension of the parameter.

Recall that $Y,U\in\rr^{(N+1)\times m}$, but now random quantities. For each fixed $m,N$, Algorithm~\ref{algorithm:h} can be used to find the optimizing $\hat h^{N,m}$, which now becomes a random vector in $\rr^{q+1}_+$. We will give limit results on consistency and asymptotic normality for the $\hat h^{N,m}$ in three cases.
First for $m\to\infty$, and the columns $U^j$ of $U$ ($j=1,\ldots,m$) form an i.i.d.\ sample. Then for $N\to\infty$, and the rows $U_i$ of $U$ ($i=0,\ldots,N$) form an i.i.d.\ sample. Finally for $N,m\to\infty$, and all $U_{ij}$ ($i=0,\ldots,N$, $j=1,\dots,m$) form an i.i.d.\ sample.


\begin{assumption}\label{assumption:delta}
We assume throughout this section the `true' relationships
\begin{equation}\label{eq:delta}
Y_{ij}=(T(h^*)U^j)_i\delta_{ij},\quad  i=0,\ldots,N, \quad j=1,\ldots,m,
\end{equation}
where $h^*$ is an interior point of $\rr^{q+1}_+$, and the $\delta_{ij}$ are nonnegative random variables, representing multiplicative noise. We will always assume that the  $\delta_{ij}$ form a double array of i.i.d.\ random variables,  that  all $\delta_{ij}$ are independent of all $U_{ij}$ and that $\ee\delta_{ij}=1$.
\end{assumption}
Further assumptions will be detailed in the subsections below.

\subsection{Asymptotics for $m\to\infty$, $N$ fixed}\label{sec:mtoinfty}

For matrices $Y,U$ one can write $\ii(Y||T(h)U)=\sum_{j=1}^m\ii(Y^j||T(h)U^j)$, with the $Y^j$ and $U^j$ the columns of the matrices $Y$ and $U$ respectively. In this section we assume, next to Assumption~\ref{assumption:delta}, that the pairs $(Y^j,U^j)$ are i.i.d.
Let $(y,u)$ be a pair of random vectors that has the same distribution as each of the $(Y^j,U^j)$. Elements of $y$ (and $u$) are denoted  $y_i$ (and $u_i$).
Here is the first result, basically the same as \cite[Lemma~27]{fs2015}.

\begin{lemma}\label{lemma:lc}
Assume the model \eqref{eq:delta}, independence of $u_i$ and $\delta_i$, $\ee u_i<\infty$, $\ee\delta_i=1$,  and $\ee\delta_i |\log\delta_i|<\infty$. Then it holds for all $h\in\ch$ that
\begin{align*}
\lefteqn{\ee\ii(y||T(h)u)} \\
& =\ee\ii(T(h^*)u||T(h)u)  + \:\sum_i(\ee(T(h^*)u)_i\ee(\delta_i\log\delta_i).
\end{align*}
Minimizing the function $h\mapsto\ee\ii(y||T(h)u)$ (referred to below as the limit criterion) is therefore equivalent to minimizing $h\mapsto\ee\ii(T(h^*)u||T(h)u)$.
\end{lemma}
The following proposition parallels \cite[Proposition~28]{fs2015} with some minor differences in the statement and the proof.

\begin{proposition}\label{prop:limcrit}
Let $\pp(u_0>0)=1$ and $\ee u_j^2<\infty$ for all $j$. The limit criterion $h\mapsto\ee\ii(y||T(h)u)$ is strictly convex on the set where it is finite (and hence on a neighbourhood of $h^*$) and has a unique minimum for $h=h^*$.
\end{proposition}

\begin{proof}
We show that the Hessian $H(h)\in\rr^{(q+1)\times(q+1)}$ at $h$ of the limit criterion is strictly positive definite on the set where the limit criterion is finite. A computation shows that the $kl$-element of this matrix is equal to (recall the convention $u_i=0$ for $i<0$)
\[
H(h)_{kl}=\ee\sum_{j=0}^N
\frac{(T(h^*)u)_j}{(T(h)u)_j^2}u_{j-k}u_{j-l}.
\]
Hence, for any vector $x=( x_0, \ldots, x_N )^\top$ one has,   using the convolution notation $(u*x)_j:=\sum_k x_ku_{j-k}$,
\begin{align*}
x^\top H(h)x
& = \ee\sum_{j=0}^N\frac{(T(h^*)u)_j}{(T(h)u)_j^2}(u*x)_j^2.
\end{align*}
Suppose that $x^\top H(h)x=0$ for some $x\in\rr^{q+1}$. Then  $\ee\frac{(T(h^*)u)_j}{(T(h)u)_j^2}(u*x)_j^2$ has to be zero for all $j$, in particular for $j\in\{0,\ldots,q\}$. Hence $\frac{(T(h^*)u)_j}{(T(h)u)_j^2}(u*x)_j^2=0$~a.s.\ for $j\in\{0,\ldots,q\}$. Since $(T(h^*)u)_j\geq h^*_ju_0$, which is strictly positive by the assumptions, one can only have $\frac{(T(h^*)u)_j}{(T(h)u)_j^2}(u*x)_j^2=0$~a.s.\ if $(u*x)_j= 0$~a.s. for all $j=0,\ldots,q$.
This gives a system of linear equations $\bar Ux=0$, where $\bar U\in\rr^{(q+1)\times(q+1)}$ is lower triangular with all diagonal elements equal to $u_0$. Using $\pp(u_0>0)=1$, we deduce that $x=0$ iff $x^\top H(h)x=0$.
From Lemma~\ref{lemma:lc} it follows that the limit criterion has a minimum at $h=h^*$, and by strict convexity this must be the unique minimizer.
\end{proof}
As in the present case $N$ is fixed, we simply write $\hat h^m$ for the estimators, i.e. the minimizers of $F_m(h)=\sum_{j=1}^m\ii(Y^j||T(h)U^j)$. The following proposition, basically the same as~\cite[Proposition~29]{fs2015}, describes the large sample behaviour of the $\hat h^m$ for the number of input sequences $m\to\infty$ and the observation horizon $N$ fixed. We include the proof for the sake of completeness.

\begin{proposition}\label{prop:consistent1}
Let Assumption~\ref{assumption:delta} be in force, in particular~\eqref{eq:delta}, and assume that the random vectors $U^j$ form an i.i.d.\ sequence.
Let $\pp(U_{0j}>0)>0$ and $\ee U_{ij}^2<\infty$ for all $i,j$, moreover assume that $h^*$ is an interior point.

The estimators $\hat{h}^m$, defined as the minimizers of the objective function $\sum_{j=1}^m\ii(Y^j||T(h)U^j)$ are consistent. Moreover, this sequence is asymptotically normal, for some positive definite $\Sigma\in\rr^{(q+1)\times (q+1)}$ we have $\sqrt{m}(\hat{h}^m-h^*)\stackrel{d}{\to} N(0,\Sigma)$.
\end{proposition}

\begin{proof}
The limit criterion $h\mapsto\ee\ii(Y||T(h)U)$ is strictly convex, continuous on the set where it is finite. Therefore from~\cite[Problem~5.27]{vandervaart} we conclude that the conditions of \cite[Theorem~5.7]{vandervaart} are satisfied and consistency follows.
To show that the estimators $\hat{h}^m$ are asymptotically normal with covariance function as given in \cite[Theorem~5.23]{vandervaart}, we have to show that the Hessian $H(h^*)$ at $h^*$ of the limit criterion is strictly positive definite. But this follows from the proof of Proposition~\ref{prop:limcrit}  taking $h=h^*$.
\end{proof}

\subsection{Asymptotics for $N\to\infty$, $m$ fixed}\label{section:ntoinfty}

The standing assumption is again Assumption~\ref{assumption:delta}. Let, as before, $Y$ and $U$ be matrices. Write $\ii(Y||T(h)U)=\sum_{i=0}^N\ii(Y_i||(T(h)U)_i)$, with the $Y_i$ and $(T(h)U)_i$ the rows of the matrices $Y$ and $T(h)U$.

We'd like to have all rows $Y_i$ mutually independent, but row $Y_i$ partly uses the same inputs as  $Y_{i+1}$, namely the rows $U_{i+1},\ldots,U_{i-q+1}$ (for $i\geq q$). Consider the rows $U_i$ and $U_{i+q+1}$. The elements of these rows that are needed to compute the convolutions $(T(h)U)_{ij}$ are $U_{ij},\ldots,U_{i-q,j}$, whereas for $(T(h)U)_{i+q+1,j}$ one needs $U_{i+q+1,j},\ldots,U_{i+1,j}$. We see that these sets of elements have empty intersection. To have independence of rows $(T(h)U)_i$ and $(T(h)U)_{i+q+1}$  we will assume that  the  rows $U_{i}$ are independent. The interpretation is that in the collective experiments, at different times independent row vectors are used as inputs.

\begin{lemma}\label{lemma:n}
Assume that the rows $U_{i}$ form an i.i.d.\ sequence, and that all necessary expectations are finite. Consider the random criterion function
\[
I_N(h)=\frac{1}{N}\sum_{i=0}^N\ii(Y_i||(T(h)U)_i).
\]
Then one has, for $N\to\infty$ the a.s.\ convergence
\[
I_N(h)\to \ee\ii(Y_{q}||(T(h)U)_{q}).
\]
\end{lemma}

\begin{proof}
We split the sum $\sum_{i=0}^N\ii(Y_i||(T(h)U)_i)$ into the $q+1$ sums
\[
\sum_{i=0}^{\lfloor \frac{N}{q+1}\rfloor}\ii(Y_{i(q+1)+l}||(T(h)U)_{i(q+1)+l})
\]
and a remainder term of at most $q$ terms $\ii(Y_i||(T(h)U)_i)$. The remainder term divided by $N$ tends to zero a.s. For each $l$ the strong law applies because of the independence properties and we have the a.s.\ convergence
\[
\frac{1}{N}\sum_{i=0}^{\lfloor \frac{N}{q+1}\rfloor}\ii(Y_{i(q+1)+l}||(T(h)U)_{i(q+1)+l})\to\frac{1}{q+1}\ee \ii(Y_{q+1+l}||(T(h)U)_{q+1+l}).
\]
Since the rows $(Y_i,(T(h))U_i)$ have the same distribution for all $i\geq  q$, one has the identity $\ee \ii(Y_{q+1+l}||(T(h)U)_{q+1+l})=\ee \ii(Y_{q}||(T(h)U)_{q})$. The result follows.
\end{proof}

\begin{proposition}\label{prop:limcrit1}
Assume the model \eqref{eq:delta}, the rows $U_i$ form an i.i.d.\ sequence, and for all $i,j$, $\ee U_{ij}<\infty$, $\ee\delta_{ij}=1$,  and $\ee\delta_{ij}|\log\delta_{ij}|<\infty$. Then it holds that
\begin{align*}
\lefteqn{\ee\ii(Y_q||(T(h)U)_q)} \\
& =\ee\ii((T(h^*)U)_q||(T(h)U)_q))  + \:\sum_{j}(\ee(T(h^*)U)_{qj}\ee(\delta_{qj}\log\delta_{qj}).
\end{align*}
Moreover, the divergences $\ee\ii((T(h^*)U)_i||(T(h)U)_i)$ are identical for all $i\geq q$ and the limit criterion $h\mapsto\ee\ii(Y_q||(T(h)U))_q)$ is strictly convex, and hence continuous, on the set where it is finite (and hence on a neighbourhood of $h^*$). It has a unique minimum for $h=h^*$, if $\pp(U_{0j}>0)>0$ for at least one $j$ and $\ee U_{0j}^2<\infty$ for all $j$.
\end{proposition}
\begin{proof}
The proof of the first assertion is like the one of Lemma~\ref{lemma:lc}. The second assertion follows from the observation that for $i\geq q$ in the computation of the divergence, one needs $q+1$ inputs $U_{i},\ldots,U_{i-q}$ and these have identical distributions. Strict convexity and uniqueness are proved in the same way as for Proposition~\ref{prop:limcrit}.
\end{proof}
Minimizing the function $h\mapsto\ee\ii(Y_q||(T(h)U))_q)$ (referred to below as the limit criterion) is thus equivalent to minimizing $h\mapsto\ee\ii((T(h^*)U)_q||(T(h)U))_q)$.

As in the present case $m$ is fixed, we write $\hat h^N$ for the estimators. The following proposition describes the large sample behaviour of the $\hat h^N$ for $N\to\infty$.

\begin{proposition}\label{prop:consistent2}
Let Assumption~\ref{assumption:delta} be in force, in particular~\eqref{eq:delta}, and assume the rows $U_i$ form an i.i.d.\ sequence. Let $\pp(U_{0j}>0)>0$ for at least one $j$ and $\ee U_{0j}^2<\infty$ for all $j$, moreover assume that $h^*$ is an interior point.

The estimators $\hat{h}^N$, defined as the minimizers of the objective function $\sum_{i=0}^N\ii(Y_i||(T(h)U)_i)$ are consistent. Moreover, this sequence is asymptotically normal, for some positive definite $\Sigma\in\rr^{(q+1)\times (q+1)}$ we have $\sqrt{N}(\hat{h}^N-h^*)\stackrel{d}{\to} N(0,\Sigma)$.
\end{proposition}

\begin{proof}
As the convolutions are not independent anymore, we cannot immediately follow the same path as in the proof of Proposition~\ref{prop:consistent1}. Still, the key to prove the result in the present case is the independence of the rows $U_i$ and that the $\delta_{ij}$ are independent.

Recall that a sequence of random variables or vectors $X_i$ is $q$-dependent if for every possible time index $t$ the (possibly infinite) sequences $(\ldots,X_{t-1},X_t)$ and $(X_{t+1+q}, X_{t+2+q},\ldots)$ are independent, and that a $q$-dependent sequence is automatically strong mixing. It follows from the assumptions that the $(Y_k,(T(U))_k)$ are $q$-dependent and so are the $\ii(Y_k||(T(h)U)_k)$, which then trivially become a strong mixing sequence. Hence, one can apply Ibragimov's central limit theorem~\cite{ibragimov1975} for strongly mixing stationary sequences to have $\sqrt{N}(I_N(h)-\ee I_N(h))$ converging to a zero mean normal distribution. The asymptotic normality result for the estimators follows by a Taylor argument for  M-estimators combined with the laws of large numbers and the CLT result for the $I_N(h)$ above (see~\cite[pages 51 and 72]{vandervaart}), or by application of the delta-method. See also \cite[Chapters 5 and 19]{vandervaart}, in particular the proofs of the general Theorems~5.21 and~5.23, and \cite[Section~5.6]{vandervaart} with results on the `classical case'. To verify the consistency condition in these theorems, one needs strict convexity and continuity of the limit criterion $h\mapsto\ee\ii(Y_q||(T(h)U))_q)$ and uniqueness of its minimizer, similar to Proposition~\ref{prop:limcrit}. From~\cite[Problem~5.27]{vandervaart} one concludes that the conditions of \cite[Theorem~5.7]{vandervaart} are satisfied and consistency follows.
\end{proof}

\subsection{Asymptotics for $N,m\to\infty$}\label{section:nmtoinfty}

In this section we study the large sample behavior of the estimators $h^{N,m}$ when both the time horizon $N$ and the number of experiments $m$ may tend to infinity.
The model is again \eqref{eq:delta} and next to Assumption~\ref{assumption:delta} in this section it is additionallly assumed that both the $U_{ij}$ and the $\delta_{ij}$ are i.i.d.\ arrays with the relevant expectations finite.

We look again at the limit criteria of Lemma~\ref{lemma:lc} and Lemma~\ref{lemma:n}.
The first limit criterion becomes $\sum_{i=0}^N\ee \ii(Y_{ij}||(T(h)U)_{ij})$, with $j$ arbitrary, which equals
\[
L^1_N:=\sum_{i=0}^{q-1}\ee \ii(Y_{ij}||(T(h)U)_{ij})+(N+1-q)\ee \ii(Y_{qj}||(T(h)U)_{qj})
\]
by the assumed identity in distribution. The second limit criterion we can write as $\sum_{j=1}^m\ee\ii(Y_{qj}||(T(h)U)_{qj})$, equal to
\[
L^2_m:=m\ee\ii(Y_{qj}||(T(h)U)_{qj})
\]
by the assumed independence for this case. We see that $\lim_{N\to\infty}\frac{1}{N}L^1_N=\lim_{m\to\infty}\frac{1}{m} L^2_m=\ee\ii(Y_{qj}||(T(h)U)_{qj})$, $j$ arbitrary, for instance $j=1$. This motivates the next result.

\begin{lemma}\label{lemma:nm}
Consider the random criterion function
\[
I_{N,m}(h)=\frac{1}{Nm}\sum_{i=0}^N\sum_{j=1}^m\ii(Y_{ij}||(T(h)U)_{ij}).
\]
Then one has, for  $N,m\to\infty$ the convergence in probability
\begin{equation}\label{eq:inm}
I_{N,m}(h)\to \ee\ii(Y_{q1}||(T(h)U)_{q1}),
\end{equation}
which has $h^*$ as its unique minimizer.
\end{lemma}

\begin{proof}
For  each $j$ the random variables $\ii(Y_{ij}||(T(h)U)_{ij})$ are $q$-dependent and hence the variance of $\sum_{i=0}^N\ii(Y_{ij}||(T(h)U)_{ij})$ can be shown to be (finite and) of order $N$. As the latter sums are i.i.d.\ for different $j$ the result is that the variance of the double sum $\sum_{i=0}^N\sum_{j=1}^m\ii(Y_{ij}||(T(h)U)_{ij})$ is of order $Nm$. Hence, Chebychev's inequality gives the result on the convergence.

The minimizing property of $h^*$ follows as in the proof  of Proposition~\ref{prop:limcrit}, using the additive decomposition of the limit  in \eqref{eq:inm} into $(T(h^*)U)_{q1})||(T(h)U)_{q1})$ and a remainder term not involving $h$.
\end{proof}

\begin{remark}
Let $\rho_{N,m}=\frac{N}{m}$ and $\rho=\lim_{N,m\to\infty}\rho_{N,m}$ (assumed to exist). If $\rho=0$, the limit in Lemma~\ref{lemma:nm} coincides with the result for fixed $N$, if $\rho=\infty$ one retrieves the result of Lemma~\ref{lemma:n}, since under the present independence assumptions $\ii(Y_{q}||(T(h)U)_{q})=m\ii(Y_{q1}||(T(h)U)_{q1})$.
\end{remark}

\begin{proposition}\label{prop:consistent3}
Let Assumption~\ref{assumption:delta} be in force, in particular~\eqref{eq:delta}, and assume all $U_{ij}$ form an i.i.d.\ double array. Let $\pp(U_{ij}>0)>0$  and $\ee U_{ij}^2<\infty$ for all $i,j$, moreover assume that $h^*$ is an interior point. Let $N,m\to\infty$. The estimators $\hat{h}^{N,m}$, defined as the minimizers of the objective function $I_{N,m}(h)$ are consistent. Moreover, this sequence is asymptotically normal, for some positive definite $\Sigma\in\rr^{(q+1)\times (q+1)}$ we have $\sqrt{Nm}(\hat{h}^{N,m}-h^*)\stackrel{d}{\to} N(0,\Sigma)$.
\end{proposition}

\begin{proof}
For consistency one needs Lemma~\ref{lemma:nm} and uniqueness of the minimizer of the expectation in~\eqref{eq:inm}. The remainder follows as in the proof of Proposition~\ref{prop:consistent2}, using for any fixed $j$ the $q$-dependence of the $Y_{ij}$, $i\geq 0$ and the independence, for fixed $i$ of the $Y_{ij}$, $j\geq 1$. \end{proof}

\subsection{Misspecified models}

The standing assumption in this section until now was Assumption~\ref{assumption:delta} that postulated the existence of a `true' parameter $h^*$. In absence of this assumption we have the following counterpart of Proposition~\ref{prop:limcrit} under the conditions of Section~\ref{sec:mtoinfty}. Similar results holds for the situations of Sections~\ref{section:ntoinfty} and~\ref{section:nmtoinfty}.

\begin{proposition}\label{prop:limcritmis}
Let $\pp(u_0>0)=1$ and $\ee u_j^2<\infty$ for all $j$. The limit criterion $h\mapsto\ee\ii(y||T(h)u)$ is strictly convex on the set where it is finite and has a unique minimum. The unique minimizer coincides with $h^*$ when Assumption~\ref{assumption:delta} holds.
\end{proposition}

\begin{proof}
The proof follows the lines of the proof of Proposition~\ref{prop:limcrit}, but in the computation of the Hessian one has to replace the quantities $(T(h^*)u)_j$ with $Y_j$. The Hessian is then again seen to be positive definite, and the existence of a unique minimum follows.

That this minimizer coincides with $h^*$ under Assumption~\ref{assumption:delta}, follows from Lemma~\ref{lemma:lc}. 
\end{proof}
Calling the unique minimizer $h^*$, one obtains that all previous results on consistency --under the presented conditions-- continue to hold with the `true' parameter replaced with this $h^*$. The same is true for the results for asymptotic normality. See Example~5.25 in \cite{vandervaart} for a similar discussion on maximum likelihood estimation for misspecified models. In that situation the `true' parameter is replaced by the one that minimizes the Kullback-Leibler information between the distribution of the data and the distribution given by the misspecified model. The analogy with our setting is obvious.

\section{Numerical experiments}\label{section:numerics}

In this section we provide the results of a number of numerical experiments that illustrate the behaviour of Algorithm~\ref{algorithm:h}. All figures can be found at the end of the paper.
We have observed experimentally that usually the iterative algorithm converges very fast in many instances, which is illustrated by the examples. In many cases 50 iterations would have sufficed. For the sake of graph readability in the examples reproduced here the order $q$ has been limited to 5, leading to a parameter vector $h$ of length 6. Each of the graphs shows the iterates $h^t_k$ ($k=0,\dots,5$)  with the iteration number $t$ on the horizontal axis, and the 6 values of the impulse response $h_k^t$ on the vertical axis, different colors representing the different $k$'s. As another simplification in the graphs we sometimes omit the first iterates.
In Figures~\ref{fig:true1}--\ref{fig:noisy4}
the diamonds at the right end of the graph indicate the true $h^*$ target values. In all cases the $U_{ij}$ are generated as independent uniform $U(0.1,10)$ random variables. The precise features underlying the different experiments are further detailed  below. The different experiments highlight the role of the parameters $m$ and $N$, especially when the system is observed with noise. Relatively small values of $m$ compared to high values of $N$ give satisfactory results. For the asymptotics of Proposition~\ref{prop:consistent3} it is important that only the product $Nm$ is large.

In the first two examples, Figure~\ref{fig:true1} and Figure~\ref{fig:true2}, we investigate whether the algorithm is capable of retrieving the \emph{true} parameter vector $h^*$, when the output data are actually generated by $h^*$. After that we investigate the behavior of the algorithm when we have noisy observations of the output.
Here we are in the statistical setting of Section~\ref{section:stats}. The $\delta_{ij}$ are taken as $\exp(Z_{ij}/10-1/200)$, where the $Z_{ij}$ are independent standard normal random variables. Note that indeed $\ee \delta_{ij}=1$. Figures~\ref{fig:noisy1}--\ref{fig:noisy4} illustrate the large sample behaviour of the estimators. We see that for not too large values of $N$, already moderate values of $m$ give good results, this illustrates Proposition~\ref{prop:consistent3}. For small values of $m$, e.g.\ $m=1$, one needs a relatively large value of $N$ to have satisfactory results. This is probably partly due to the dependence between rows of $Y$.
In the last examples the input/output relation generating the outputs is that of an arbitrary positive system. In this case the $h$ generated by the algorithm is the impulse response of the best convolutional system approximation to the given system. Figures~\ref{fig:arb1} and~\ref{fig:arb2} also illustrate Remark~\ref{remark:boundary} on boundary solutions.

\section{Conclusions}

We posed the nonparametric  approximation problem for scalar nonnegative input/output systems via impulse response convolutions of finite order, based on multiple observations of input/output signal pairs. The problem is  converted into a nonnegative matrix factorization with special structure for which we used Csisz\'ar's I-divergence as the criterion of optimality. Conditions have been given that guarantee the existence and uniqueness of the minimum. An algorithm whose iterates converge to the unique minimizer has been presented. For the case of noisy observations of a true system we also proved the consistency of the parameter estimators under different large sample regimes (many observation times, many inputs, or a mix of these). Numerical experiments confirm the asymptotic results and often exhibit fast convergence to the  minimizer of the objective function.


\bibliographystyle{plain}
%


\newpage

\begin{figure}
\begin{center}
	\includegraphics[width=0.9\textwidth]{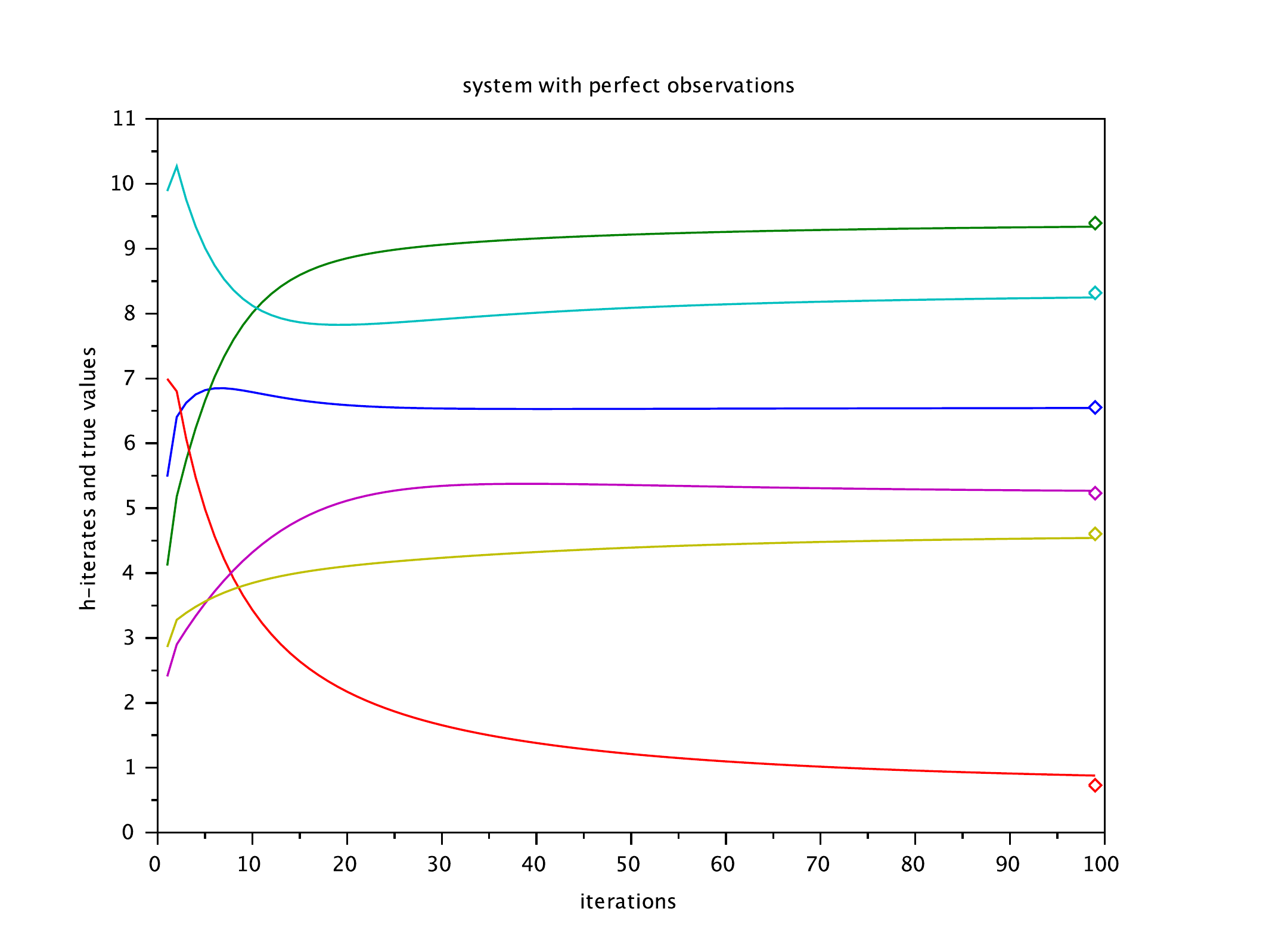}
	\vspace{-3ex}
		\caption{{\rm noiseless observations, $m=5$, $N=10$}
		\label{fig:true1}}
\end{center}
\end{figure}

\begin{figure}
\begin{center}
	\includegraphics[width=0.9\textwidth]{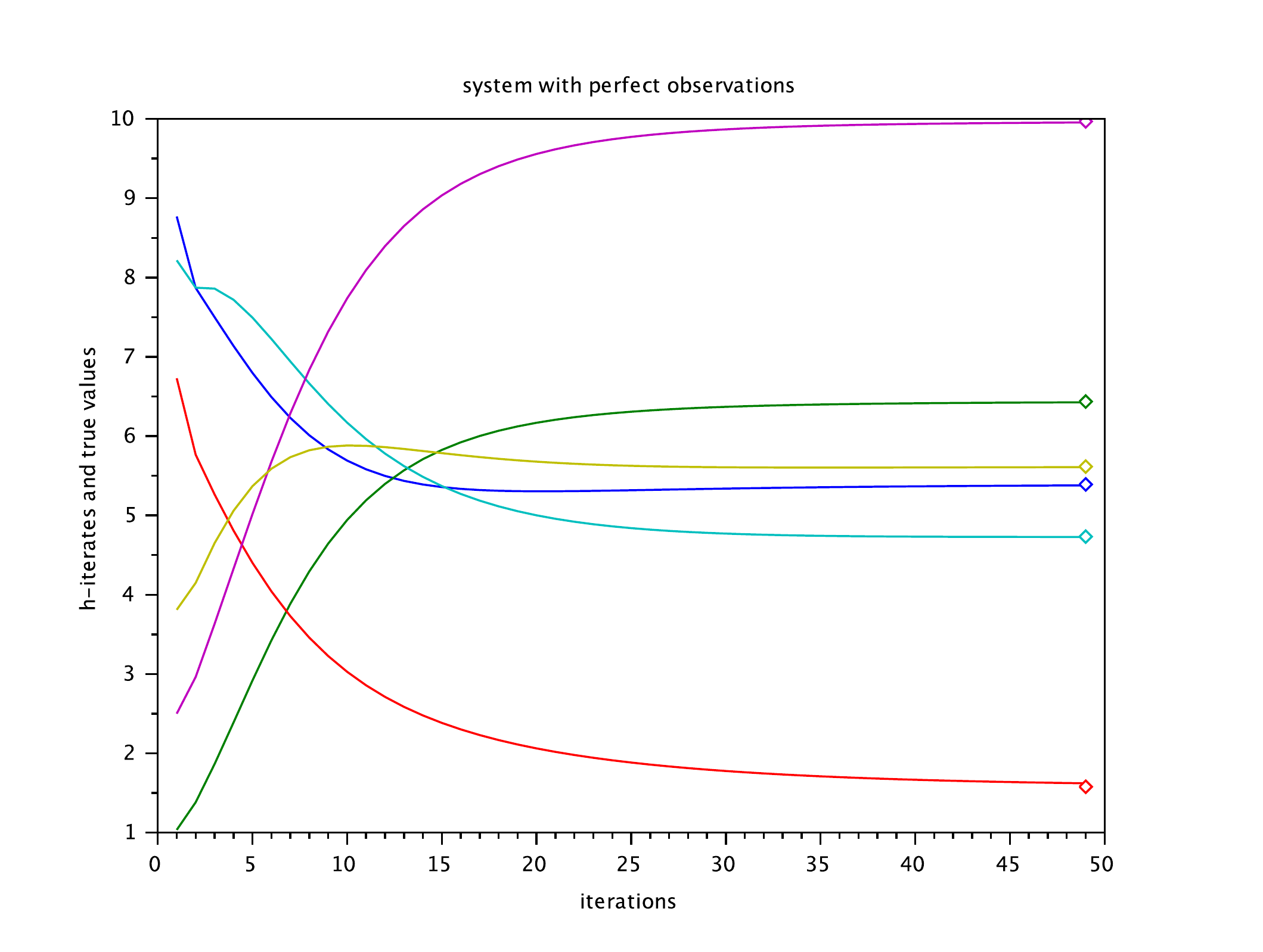}
	\vspace{-3ex}
		\caption{{\rm noiseless observations, $m=10$, $N=5$}
		\label{fig:true2}}
\end{center}
\end{figure}

%


\begin{figure}
\begin{center}
	\includegraphics[width=0.9\textwidth]{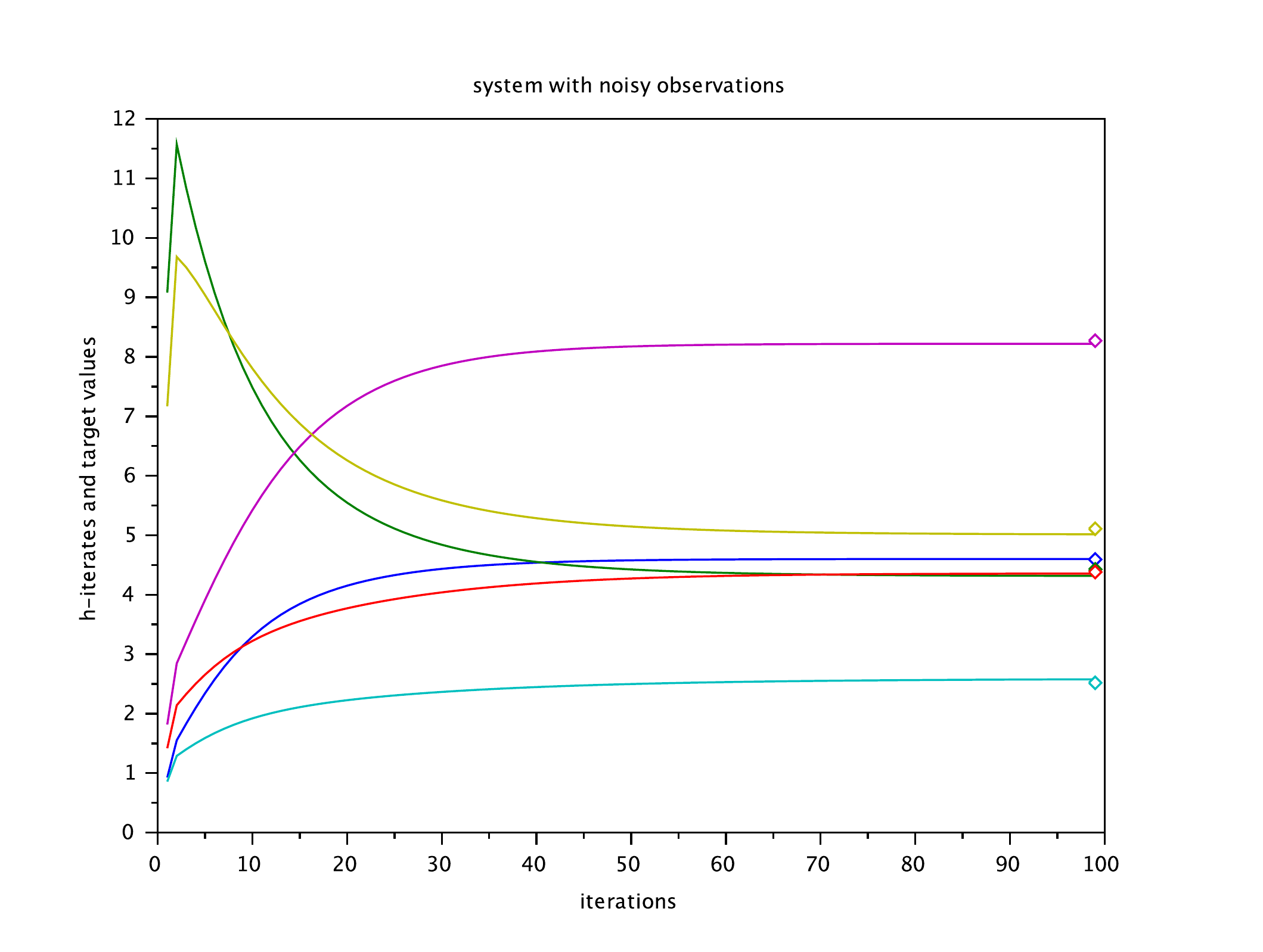}
	\vspace{-3ex}
		\caption{{\rm noisy observations, $m=30$, $N=20$}
		\label{fig:noisy1}}
\end{center}
\end{figure}

\begin{figure}
\begin{center}
	\includegraphics[width=0.9\textwidth]{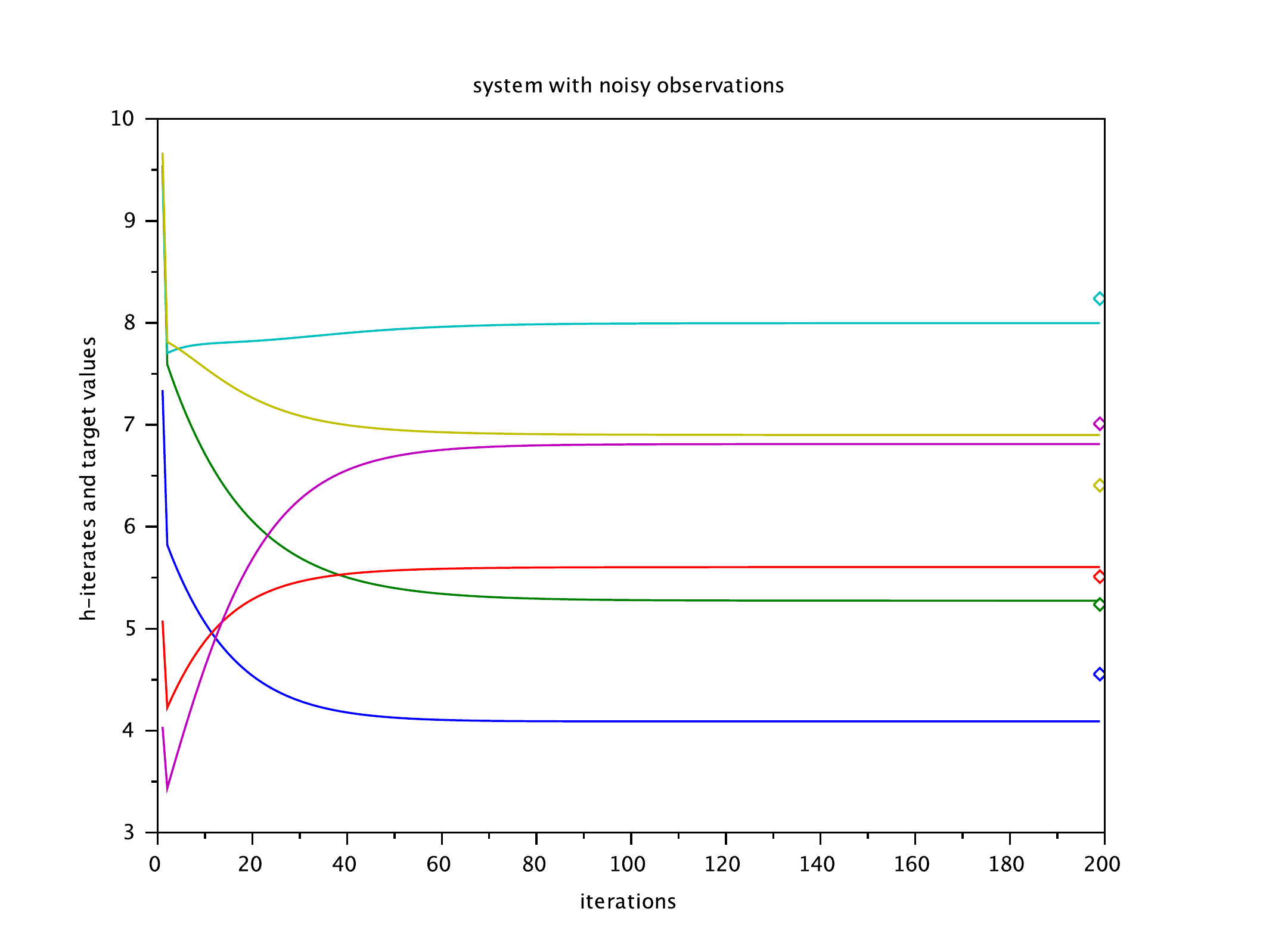}
	\vspace{-3ex}
		\caption{{\rm noisy observations, $m=1$, $N=100$}
		\label{fig:noisy2}}
\end{center}
\end{figure}

\begin{figure}
\begin{center}
	\includegraphics[width=0.9\textwidth]{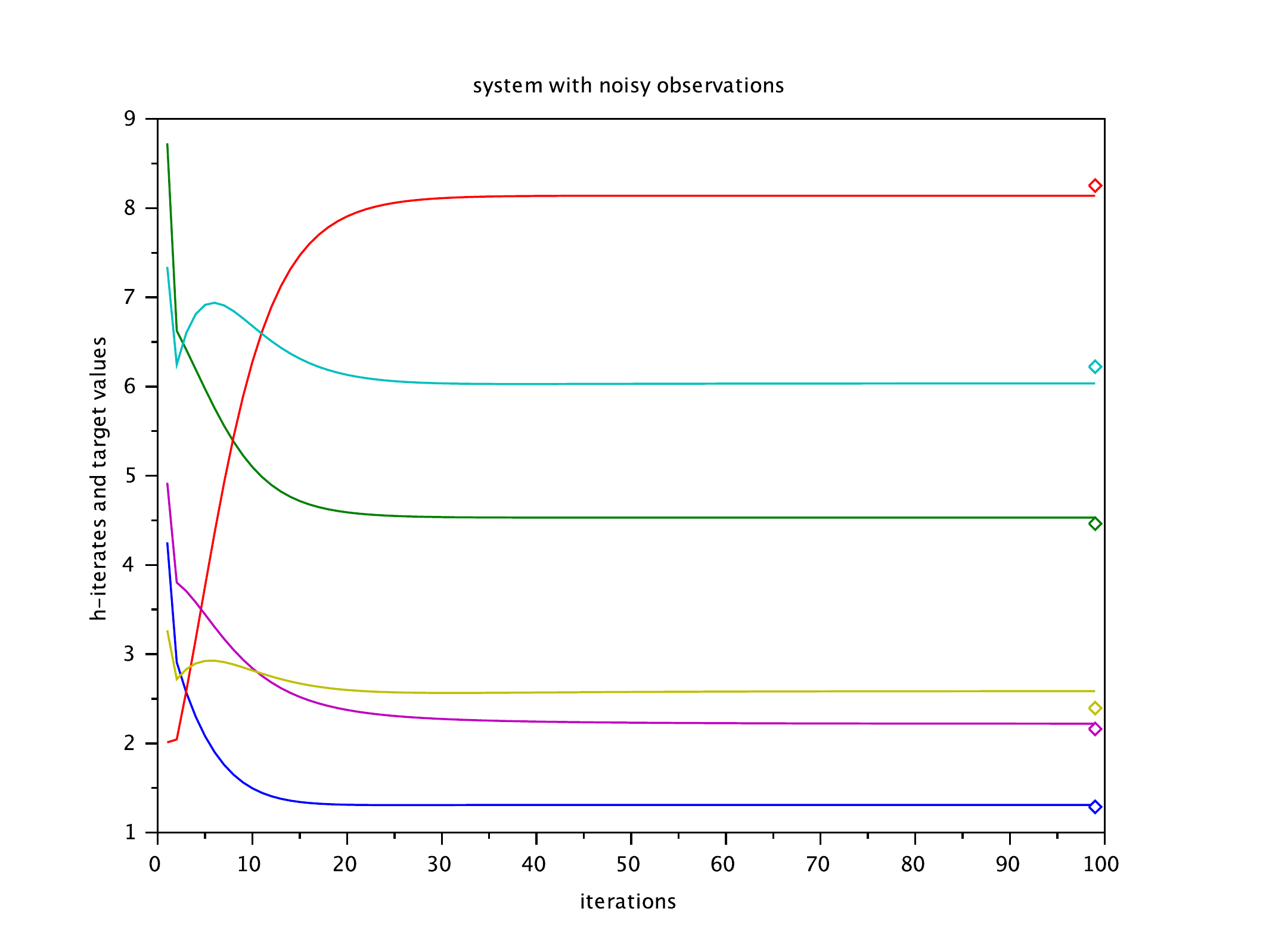}
	\vspace{-3ex}
		\caption{{\rm noisy observations, $m=30$, $N=5$}
		\label{fig:noisy3}}
\end{center}
\end{figure}

\begin{figure}
\begin{center}
	\includegraphics[width=0.9\textwidth]{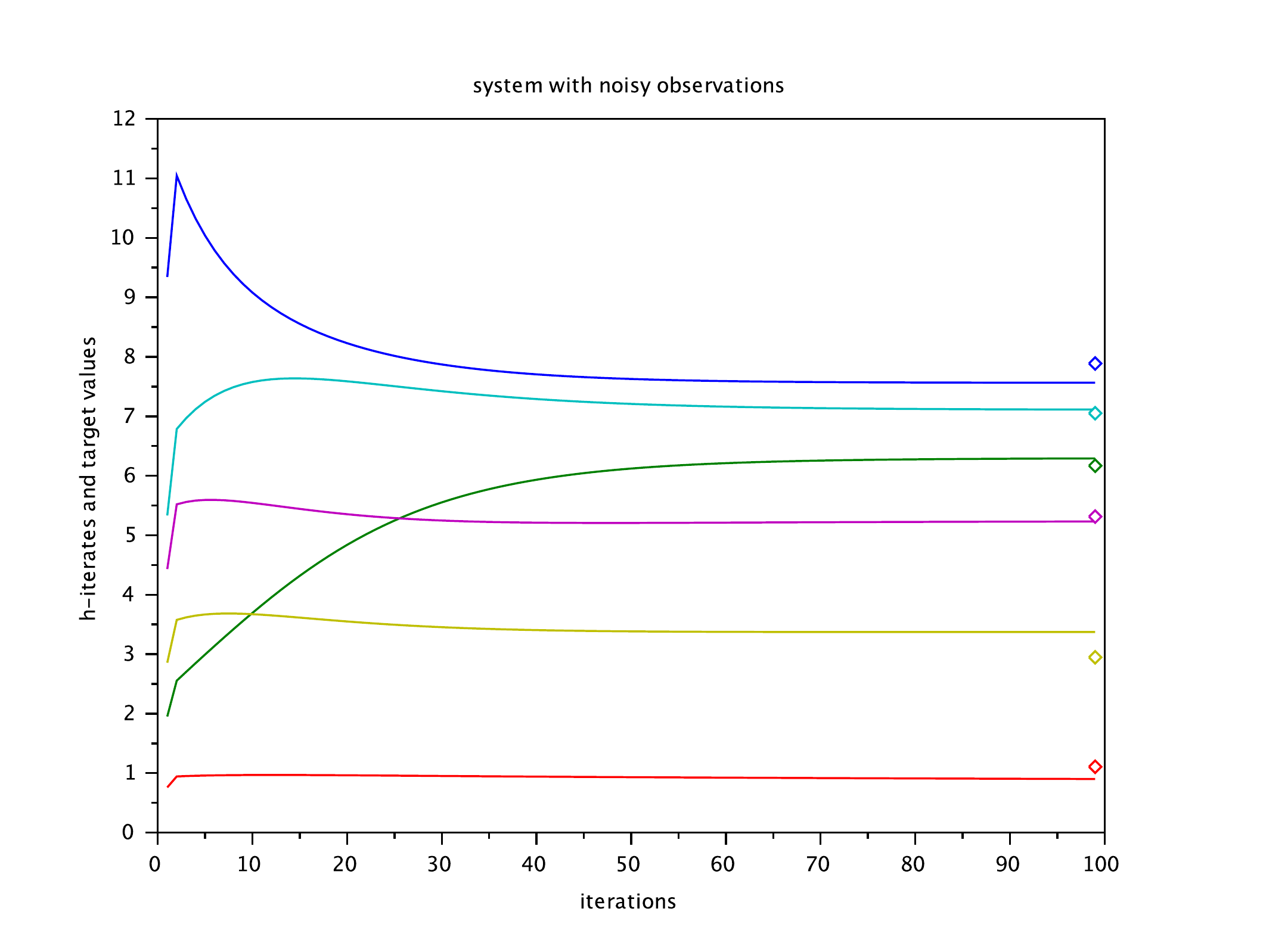}
	\vspace{-3ex}
		\caption{{\rm noisy observations, $m=5$, $N=30$}
		\label{fig:noisy4}}
\end{center}
\end{figure}

\begin{figure}
\begin{center}
	\includegraphics[width=0.9\textwidth]{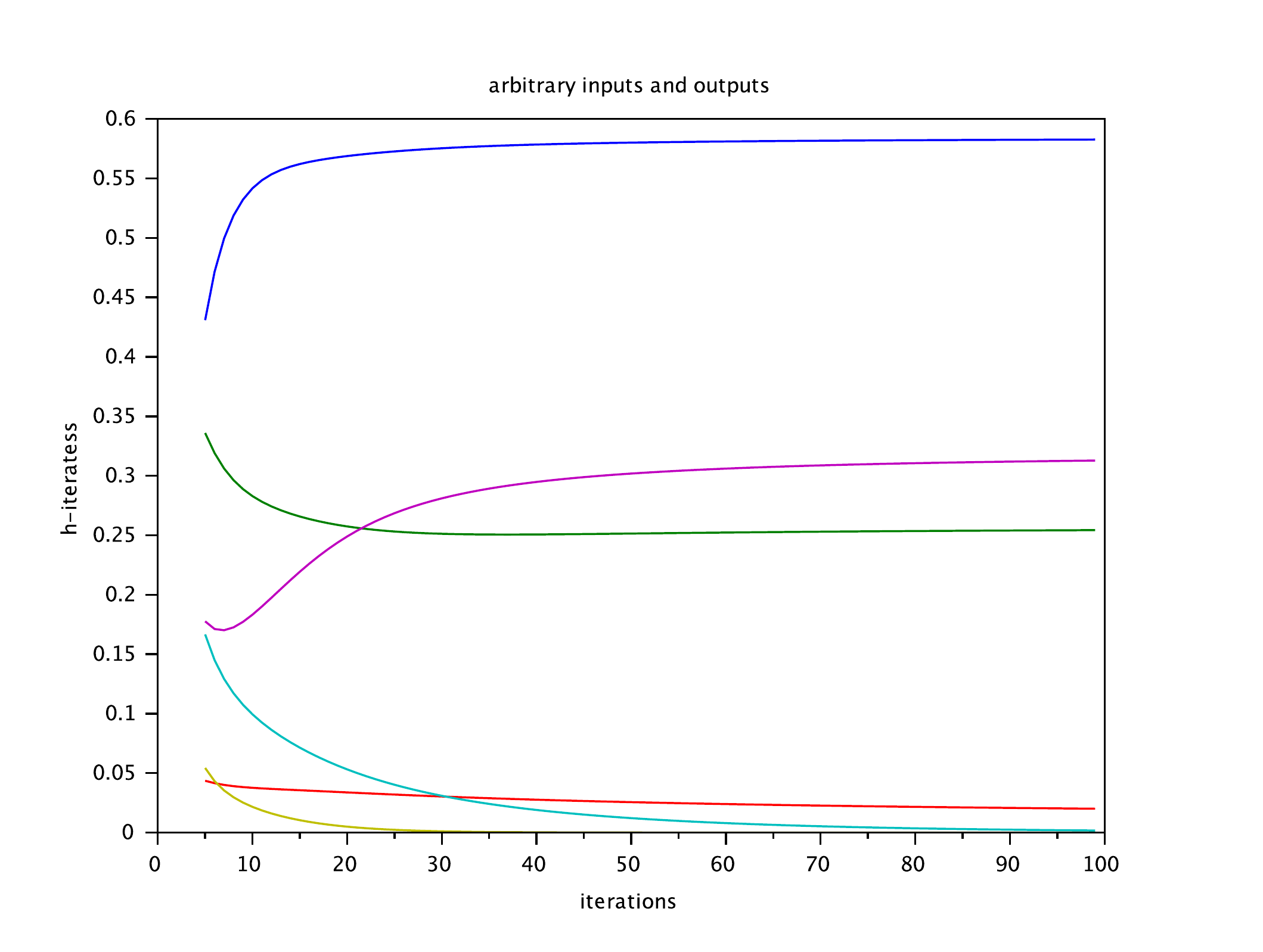}
	\vspace{-3ex}
		\caption{{\rm arbitrary system, $m=10$, $N=8$}
		\label{fig:arb1}}
\end{center}
\end{figure}

\begin{figure}
\begin{center}
	\includegraphics[width=0.9\textwidth]{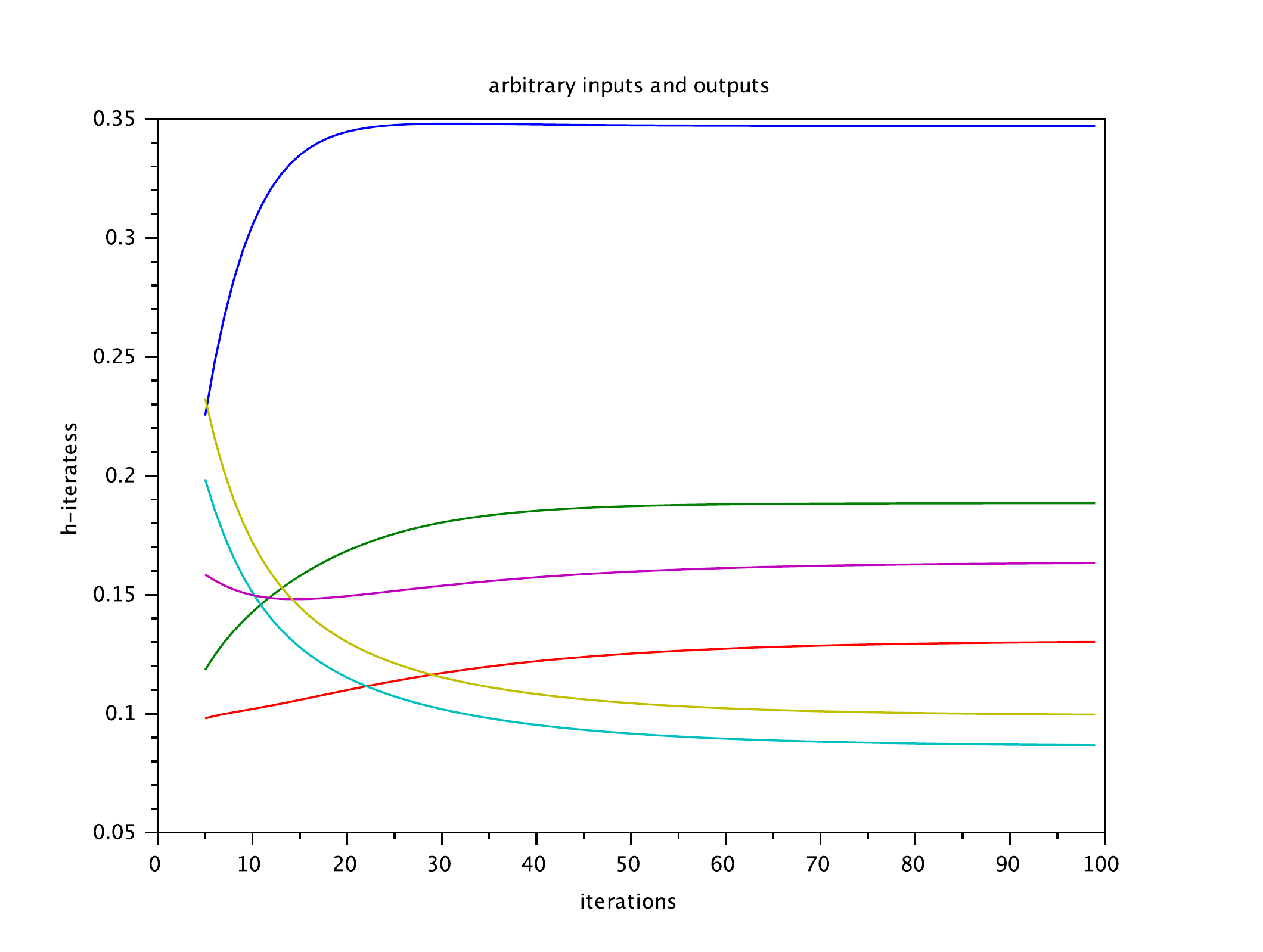}
	\vspace{-3ex}
		\caption{{\rm arbitrary system, $m=10$, $N=50$}
		\label{fig:arb2}}
\end{center}
\end{figure}

\end{document}